\documentclass[a4paper]{article}
\usepackage{amsthm,amssymb,amsmath,enumerate,mathtools,tikz,tikz-3dplot,enumitem}
\theoremstyle{plain}
\usepackage{graphicx,subcaption,hyperref}
\usetikzlibrary{arrows}
\usetikzlibrary{arrows,decorations.pathmorphing,backgrounds,positioning,fit,matrix,shapes, positioning}
\usetikzlibrary{decorations.pathmorphing}
\usepackage{pgfplots}
\pgfplotsset{compat=1.11}
\usetikzlibrary{decorations.markings}
\tikzset{snake it/.style={decorate, decoration=snake}}
\usetikzlibrary{arrows,automata,positioning,snakes}
\usepackage{pst-node}
\usepackage{tikz-cd}
\tikzset{
	>=stealth',
	punkt/.style={
		rectangle,
		rounded corners,
		draw=black, very thick,
		text width=6.5em,
		minimum height=2em,
		text centered},
	pil/.style={
		->,
		thick,
		shorten <=2pt,
		shorten >=2pt,}
}
\usepackage{tkz-euclide}
\usetikzlibrary{intersections}

\usepackage{float}
\usepackage{comment}
\usepackage{marginnote}

\newcommand{\N}{\ensuremath{\mathbb{N}}}
\newcommand{\Z}{\ensuremath{\mathbb{Z}}}

\newcommand{\st}{\ensuremath{\mathsf{St}}}
\newcommand{\cn}{\ensuremath{\mathsf{cn}}}

\theoremstyle{definition}

\theoremstyle{plain}
\newtheorem{theorem}{Theorem}[section]

\newtheorem{lemma}[theorem]{Lemma}

\newtheorem{thm}[theorem]{Theorem}
\newtheorem{coro}[theorem]{Corollary}
\newtheorem{remark}[theorem]{Remark}

\newtheorem{prob}{Problem}

\newtheorem*{theorem*}{Theorem}

\theoremstyle{remark}

\title{Splitting groups with cubic Cayley graphs of connectivity two}
\author{Babak Miraftab \and Konstantinos Stavropoulos \smallskip \and Department of Mathematics\\ University of Hamburg}
\date{\today}

\newenvironment{txteq*}
{
	\begin{equation*}
		\begin{minipage}[t]{0.85\textwidth} 
			\em                                
		}
		{\end{minipage}\end{equation*}\ignorespacesafterend}

\begin{document}
	\maketitle

	\begin{abstract}
		A group $G$ splits over a subgroup $C$ if $G$ is either a free product with amalgamation $A \underset{C}{\ast} B$ or an {\rm HNN}-extension $G=A \underset{C}{\ast} (t)$. We invoke Bass-Serre theory to classify all infinite groups which admit cubic Cayley graphs of connectivity two in terms of splittings over a subgroup. 
	\end{abstract}
	
	\section{Introduction}
	The study of the structure of groups in terms of the connectivity of their Cayley graphs was started by Jung and Watkins. They characterized infinite transitive graphs of connectivity one whose automorphism groups act on their vertex sets as primitive and regular permutation groups \cite{jung1977structure}. Later, Watkins \cite{watkins1976graphes} characterized all Cayley graphs of connectivity one.
	
	A topic that has been already paired with the connectivity of Cayley graphs in order to study them is the planarity of infinite Cayley graphs. 
	A finitely generated group $G$ is called \emph{planar} if it admits a generating set $S$ such that the Cayley graph $Cay(G,S)$ is planar.
	In that case, $S$ is called a \emph{planar generating set}.
	For the first time, in 1896, Maschke \cite{maschke1896representation}  characterized all finite groups admitting planar Cayley graphs.
	Infinite planar groups attracted more attention, as some of them are related to surface and Fuchsian groups \cite[Section 4.10]{zieschang2006surfaces} which play a substantial role in complex analysis, see survey \cite{zieschang2006surfaces}.
	Hamann \cite{planarhamann} uses a combinatorial method in order to show that planar groups are finitely presented.
	His method is based on tree-decompositions, a crucial tool of graph minor theory which we also utilize extensively
	in this paper.
	
	In \cite{droms}, Droms, B. Servatius, and H. Servatius characterized planar groups with low connectivity in terms of the fundamental group of the graph of groups.
	Indeed, they showed: 
	\begin{theorem*}{\rm{\cite[Theorem 4.4]{droms}}}
		If a group $G$ has planar connectivity\footnote{The
			planar connectivity $\kappa(G)$ of a planar group $G$ is the minimum connectivity of all its
			planar Cayley graphs.} $2$, then either $G$ is a finite cyclic or dihedral
		group, or it is the fundamental group of a graph of groups whose edge groups all have order
		two or less and whose vertex groups all have planar connectivity at least three. 
		In the latter
		case, the vertex groups have planar generating sets which include the nontrivial elements of
		the incident edge groups.
	\end{theorem*}
	
	Later, Georgakopoulos in \cite{Ageloscubic} determines all presentations of groups which admit planar cubic Cayley graphs with connectivity two. Furthermore, he provides partial information about the presentations of non-planar ones. Our result provides full information not only for the planar, but also the non-planar groups with cubic Cayley graphs of connectivity two.
	
	More specifically, Georgakopoulos' method does not assert anything regarding (and is, in a sense, independent of) splitting
	the group over subgroups to obtain its structure.
	By combining tree-decompositions and Bass-Serre theory, we give a simple proof for the full characterization of all groups with cubic Cayley graphs of connectivity two via the following theorem:
	\begin{thm}\label{group characterization}
		Let $G=\langle S\rangle$ be a group such that $\Gamma=\mathsf{Cay}(G,S)$ is a cubic graph of connectivity two.
		Then $G$ is isomorphic to one of the following groups:
		\begin{enumerate}[label=\rm(\roman*)]
			\item $\mathbb Z_n \ast \mathbb Z_2=\langle a,b\mid b^2,(ba)^n\rangle$ or $\langle a,b,c\mid a^2,b^2,c^2,(bcba)^n\rangle$,
			\item  $D_{2n} \underset{\mathbb Z_2}{\ast}(t)=\langle a,b\mid b^2,(ba^{-1}ba)^n\rangle$,
			\item  $D_{2n} \underset{\mathbb Z_2}{\ast}D_{2m}=\langle a,b,c\mid a^2,b^2,c^2,(ba)^n,(bc)^m\rangle$ or $\langle a,b,c \mid a^2, b^2, c^2, (bc)^{2n}, (a(bc)^n)^m \rangle$ or $\langle a,b,c \mid a^2, b^2, c^2, (bc)^n,(a(bc)^kb)^m \rangle$,
			\item $\mathbb Z_{2n} \underset{\mathbb Z_2}{\ast} D_{2m}=\langle a,b \mid b^2, a^{2n}, (ba^n)^m \rangle$,
			\item $D_{\infty} \underset{\mathbb Z_2}{\ast} D_{2m}=\langle a,b,c \mid a^2, b^2, c^2, (a(bc)^nb)^m \rangle$.  
		\end{enumerate}
	\end{thm}
	
	Theorem~\ref{group characterization} is a direct consequence of Theorems~\ref{Type I two gen},~\ref{Type I three gen},~\ref{Type II two gen} and~\ref{Type II three gen}, where we also discuss in detail the planarity of the corresponding Cayley graphs in each case, as well as their presentations. This allows us to obtain as a corollary the results of \cite{Ageloscubic}, as well as full presentations for the non-planar groups with cubic Cayley graphs of connectivity two. 
	
	Compared to the methods in~\cite{Ageloscubic}, we believe that our graph theoretical arguments are simplified, while we inevitably spend more time to recover the full algebraic structure of the group in terms of splitting this time around. Moreover, even though the planar part of our result can be relatively quickly recovered from~\cite{Ageloscubic} by applying Tietze transformations accordingly, such an approach usually works only provided that one knows or guesses beforehand the new desired presentation (in our case, the one that expresses the splitting of the group) in order to apply the correct Tietze transformations. By applying Bass-Serre theory, we naturally determine the structure of the group in terms of splitting avoiding the nuisance above, which was also the way we originally obtained it.
	
	\section{Preliminaries}
	Our terminology of groups and graphs is standard. We refer the reader to \cite{serre} for Bass-Serre theory and \cite{RDbook} for graph theory for any notation missing.
	
	\subsection{Graphs}
	
	Throughout this paper, $\Gamma$ always denotes a connected locally finite graph with vertex set $V(\Gamma)$ and edge set $E(\Gamma)$.
	A \emph{ray} is a one-way infinite path and a \emph{tail} of a ray is an infinite subpath of the ray.
	Two rays $R_1$ and $R_2$ are equivalent if there is no finite set $S$ of vertices such that $R_1$ and $R_2$ have tails in different components of $G\setminus S$.
	The equivalence classes of rays are called \emph{ends}. We refer the reader to surveys~\cite{diestel2011locally2,diestel2011locally} for a detailed study of the end structure of graphs.
	
	For a subset $U\subseteq V(\Gamma)$, we denote by $\Gamma[U]$ the subgraph induced by the vertices of $U$. A \emph{separation} of $\Gamma$ is an ordered pair $(A,B)$, where $A,B\subseteq V(\Gamma)$, such that $\Gamma[A]\cup \Gamma[B]=\Gamma$ and there is no edge between $A\setminus B$ and $B\setminus A$.
	The \emph{order} of $(A,B)$ is the size of $A\cap B$ and we denote it by $|(A,B)|$. If $|(A,B)|=k$, we say that $(A,B)$ is a $k$-separation.
	The set of separations of $\Gamma$ can be equipped with the following partial order:
	$(A,B)\leq (C,D)$ if $A\subseteq C$ and $B\supseteq D$.
	We say that $(A,B)$ is \emph{nested} with $(C,D)$ if $(A,B)$ is comparable to either $(C,D)$ or $(D,C)$. Otherwise, we say that the two separations \emph{cross}. We say that a vertex set $X$ separates vertex sets $U$ and $W$ if there exists a separation $(A,B)$ such that $U\subseteq A$, $W\subseteq B$ and $X=A\cap B$.
	
	Let $S$ be a set of vertices of $\Gamma$.
	The set of neighbors of $S$ is denoted by $N(S):=\bigcup_{s\in S} N(s) \setminus S$, whereas $N[S]$ denotes $S\cup N(S)$.
	A component $C$ of $G\setminus S$ is called \emph{tight} if $N(C)=S$.
	A separation $(A,B)$ is called \emph{tight} if both $A\setminus B$ and $B\setminus A$ have tight components.
	A separation $(A,B)$ \emph{distinguishes} two ends $\omega_1$ and $\omega_2$ if a ray $R_1\in\omega_1$ has a tail in $A\setminus B$ and a ray $R_2\in \omega_2$ has a tail in $B\setminus A$ or vise versa.
	Moreover, it distinguishes $\omega_1$ and $\omega_2$  \emph{efficiently} if 
	there is no separation $(C,D)$ distinguishing $\omega_1$ and $\omega_2$ such that $|(C,D)|<|(A,B)|$.
	Two ends $\omega_1$ and $\omega_2$ are \emph{k-distinguishable} if there is a separation of order at most $k$ distinguishing $\omega_1$ and $\omega_2$ efficiently.
	
	A separation is \emph{splitting} if it distinguishes at least two ends efficiently.
	We note that if $(A,B)$ is splitting, then $(A,B)$ is a tight separation.
	Let $(A,B)$ be a splitting $k$-separation.
	The \emph{crossing number} $\cn(A,B)$ of $(A,B)$ is the cardinality of the set containing all crossing tight~$\ell$-separations distinguishing at least two ends, where~${\ell\leq }k$ (which can be seen to be finite~\cite{t.d}).

	Let $\Gamma$ be an arbitrary connected graph.
	A \emph{tree-decomposition} of $\Gamma$ is a pair $(T , \mathcal V )$ of a tree $T$ and a family $\mathcal V = (V_t)_{t\in V(T)}$ of vertex sets $V_t\subseteq V(\Gamma)$, which are called \emph{parts}, such that:
	\begin{enumerate}[label=(T{\arabic*})]
		\item  $V (\Gamma) = \bigcup_{t\in T} V_t$,
		\item for every edge $e\in E(\Gamma)$, there exists a $t\in V(T)$ such that both ends of $e$ lie in $V_t$,
		\item $V_{t_1} \cap V_{t_3} \subseteq V_{t_2}$ whenever $t_2$ lies on the $(t_1,t_3)$-path in $T$.
	\end{enumerate}
	
	\noindent
	In order to distinguish them from the vertices of the graph $\Gamma$, we will usually refer to the vertices of the underlying tree $T$ of a tree decomposition of $\Gamma$ as nodes.
	
	An \emph{adhesion set} of $(T , \mathcal V )$ is a set of the form $V_t\cap V_{t'}$, where $tt'\in E(T)$. 
	The \emph{adhesion} of $(T , \mathcal V )$ is the maximum over the sizes of its adhesion sets.
	It is not hard to see that each adhesion set leads to a separation of $\Gamma$.
	More precisely, assume that $T_t$ and $T_t'$ are the components of $T - tt'$ containing $t$ and $t'$ respectively. 
	Then the adhesion set $V_t\cap V_{t'}$ induces the separation $(W_{t\setminus t'},W_{t'\setminus t})$ of $\Gamma$, where $W_{t\setminus t'}=\bigcup_{s\in T_t} V_s$ and $W_{t'\setminus t}=\bigcup_{s\in T_{t'}} V_s$. When every such separation is tight, we call the tree-decomposition \emph{tight} as well. Finally, a tree decomposition is \emph{reduced} if no part is contained in another one.
	
	The following folklore fact about tree decompositions and nested set of separations is well-known, see \cite{confing}.
	
	\begin{remark}\label{nested-td}
		Every nested set $\mathcal N$ of separations gives rise to a reduced tree-decomposition whose adhesion sets are exactly the elements of $\mathcal N$.
		On the other hand, each adhesion set of a tree-decomposition induces a separation and the set of all induced separations of adhesion sets of a tree-decomposition is a nested set of separations.
	\end{remark}
	
	
	
	\noindent
	
	Let $\Gamma$ be a locally finite graph with a tree-decomposition $(T,\mathcal V)$.
	We call the \emph{torso} of a part $V_t$ the supergraph of $\Gamma[V_t]$ obtained by adding to it all possible edges in the adhesion sets incident to $V_t$.
	The following general lemma for tree-decompositions is folklore.
	\begin{lemma}\label{Torso Lemma}
		Let $(T, \mathcal V)$ be a tree-decomposition of a connected graph $\Gamma$ and $t\in V(T)$ such that every adhesion set of $t$ induces a connected subgraph. Then $\Gamma[V_t]$ is connected. In particular, the torso of every part of $(T, \mathcal V)$ is connected.
	\end{lemma}
	
	In this paper, we are studying groups admitting cubic Cayley graphs of connectivity two.

	
	\subsection{Groups}
	
	Let $G$ be a group acting on a set $X$.
	Then the \emph{setwise stabilizer} of a subset $Y$ of $X$ is the set of all elements $g\in G$ stabilizing $Y$ setwise,~i.e
	\[
	\mathsf{St}_G(Y):=\{g\in G\mid gY=Y\}.
	\]
	Let $G$ be a group acting on a graph $\Gamma$.
	Then this action induces an action on $E(\Gamma)$.
	We say that $G$ acts without \emph{inversion} on $\Gamma$ if $g(uv)\neq vu$ for all $uv\in E(\Gamma)$ and $g\in G$. In the case that $g(uv)= vu$, we say that $g$ \emph{inverts} $u,v$. Notice that when $G$ acts transitively with inversion on the set $E(T)$ of edges of a tree $T$ without leaves, it must also act transitively on the set $V(T)$ of its vertices.
	
	Let $G_1=\langle S_1\mid R_1\rangle$ and $G_2=\langle S_2\mid R_2\rangle$ be two groups.
	Suppose that a subgroup $H_1$ of $G_1$ is isomorphic to a subgroup $H_2$ of $G_2$, say via an isomorphic map $\phi\colon H_1\to H_2$.
	The  \emph{free-product with amalgamation} of $G_1$ and $G_2$ over $H_1$ is
	$$G_1 \underset{H_1}{\ast} G_2=\langle S_1\cup S_2\mid R_1\cup R_2\cup h\phi(h)^{-1},\forall h\in H_1\rangle. $$
	
	If $H_1$ and $\phi(H_1)$ are isomorphic subgroups of $G_1$, then the \emph{{\rm HNN}-extension} of $G_1$ over $H_1$ with respect to $\phi$ is $$G_1\underset{H_1}{\ast}(t)=\langle S_1, t\mid R_1\cup tht^{-1}\phi(h)^{-1}, \forall h\in H_1\rangle.$$
	
	The crux of Bass-Serre theory is captured in the next lemma which determines the structure of groups acting on trees.
	\begin{lemma}\label{bass-serre}{\rm\cite{serre}} 
		Let~$G$ act without inversion on a tree which has no vertices of degree one and let $G$ act transitively on the set of (undirected) edges. 
		If $G$ acts transitively on the vertices of the tree, then $G$ is an {\rm HNN}-extension of the stabilizer of a vertex over the stabilizer of an edge. 
		If there are two orbits on the vertices of the tree, then $G$ is the free product of the stabilizers of two adjacent vertices with amalgamation over the stabilizer of an edge.
	\end{lemma}
	
	\noindent
	There is a standard way to deal with the case where we cannot apply Lemma~\ref{bass-serre} directly when $G$ acts with inversion on the tree.
	
	\begin{lemma}\label{subdivided bass-serre}
		Let~$G$ act transitively with inversion on the edges of a tree $T$ without leaves.
		Then $G$ is the free product of the stabilizer of a vertex and the stabilizer of an edge with amalgamation over their intersection.
	\end{lemma}
	
	\begin{proof}
		Subdivide every edge $tt'$ of $T$ to obtain a tree $T'$ and let $v_{tt'}$ be the corresponding new node. Notice that $G$ now acts transitively on $E(T')$ without inversion and with two orbits on $V(T')$. Each old node $t$ of $T$ has the same pointwise stabilizer in $T'$. Observe that for each new node $v_{tt'}$ we have $\st_G(v_{tt'})=\st_G(e)$, where $tt'=e \in E(T)$. The result follows from Lemma~\ref{bass-serre}.
	\end{proof}
	
	\noindent
	
	Finally, $\Z_n$ denotes the \emph{cyclic group of order $n$}. A finite \emph{dihedral} group is defined by the presentation $\langle a,b\mid b^2,a^n,(ba)^2\rangle$ and denoted by $D_{2n}$. Moreover, the infinite dihedral group $D_{\infty}$ is defined by $\langle a,b\mid b^2,(ba)^2\rangle$.   	
	
	\section{General structure of the tree-decomposition}
	
	Our key tool is the canonical tree decomposition $(T,\mathcal V)$ of Lemma~\ref{canonical td}, which will allow us to translate the action of $G$ on $\Gamma$ to an action of $G$ on $T$ and apply Bass-Serre theory. The lemma follows easily by the following results of \cite{t.d}, which we slightly reformulate for our needs.
	
	\begin{thm}{\rm\cite[Corollary 1.2.]{t.d}}\label{nested}
		Let $\Gamma$ be a locally finite graph with more than one end.
		For each $\ell\in \N$, there is a canonical tree-decomposition distinguishing all $\ell$-distinguishable ends efficiently. Moreover, its adhesion sets induce splitting separations of minimum crossing number.
	\end{thm}
	
	
	\noindent
	As a last preparatory step before we state and prove Lemma~\ref{canonical td}, we need the following result which ensures that our Cayley graph is multi-ended.
	
	\begin{lemma}{\rm\cite[Lemma 2.4]{babai1997growth}}\label{one-ended}
		Let $\Gamma$ be a connected vertex-transitive $d$-regular graph.
		Assume $\Gamma$ has one end.
		Then the connectivity of $\Gamma$ is $\geq 3(d + 1)/4$. 
	\end{lemma}	
	
	For the rest of the paper, we assume that $G=\langle S\rangle $ is an infinite finitely generated group such that $\Gamma=Cay(G,S)$ is cubic with connectivity two.
	
	The proof idea of the following lemma is similar to \cite[Corollary 4.3]{HLMR}.
	\begin{lemma}\label{canonical td}
		Let $G=\langle S\rangle $ be an infinite finitely generated group such that $\Gamma=Cay(G,S)$ is cubic with connectivity two.
		Then there exists a reduced tree-decomposition $(T,\mathcal V)$ with the following properties:
		\begin{enumerate}[label=\rm(\roman*)]
			\item The adhesion sets of $(T,\mathcal V)$ have size exactly $2$ and induce splitting separations of minimum crossing number.
			\item The action of $G$ on $\Gamma$ induces an action on $V(T)$ and a transitive action on the set of separations corresponding to the adhesion sets.
		\end{enumerate}
	\end{lemma}
	
	\begin{proof}
		Since $\Gamma$ is an infinite graph of connectivity exactly $2$, Lemma \ref{one-ended} implies that $\Gamma$ has at least two ends.
		Let $(T',\mathcal V')$ be a canonical tree-decomposition obtained by Theorem \ref{nested} which distinguishes all $2$-distinguishable ends and since $\Gamma$ is $2$-connected, its adhesion sets have size $2$.
		We consider the set $\mathcal S$ of all induced separations by $(T,\mathcal V)$ and choose a separation $(A,B)$ of $\mathcal S$.
		It follows from \cite[Lemma 2.1]{t.d} that the orbit of $(A,B)$ under the action of $G$ is a nested set $\mathcal N$ of separations of order $2$.
		Now by Remark~\ref{nested-td}, the set $\mathcal N$ gives rise to a tree-decomposition $(T,\mathcal V)$ which has the desired properties.
	\end{proof}
	
	Notice that the transitive action on the set of separations in Lemma~\ref{canonical td} (ii) implies at most two orbits for $\Gamma[\mathcal V]:= \{\Gamma[V_t]\mid t \in V(T)\}$ under the action of $G$. Moreover, we can translate the action of item (ii) to an action of $G$ on $T$ in the natural way (and $G$ will clearly act transitively on $E(T)$): 
	\[
	gt=t'\Leftrightarrow gV_t=V_{t'}. 
	\]
	
	Let $\mathcal N$ be a nested set of separations of order two in such a way that $\mathcal N$ gives a tree-decomposition as in Lemma \ref{canonical td}.
	It is easy to see that every 2-separation of $\Gamma$ such that $A\cap B$ is a proper subset of $A$ and $B$ distinguishes at least two ends, see \cite[Lemma 3.4]{caycomplex}.
	For an arbitrary element $(A,B)\in \mathcal N$, there are three cases in terms of the degrees of the vertices of the separator in each side of the separation:
	
	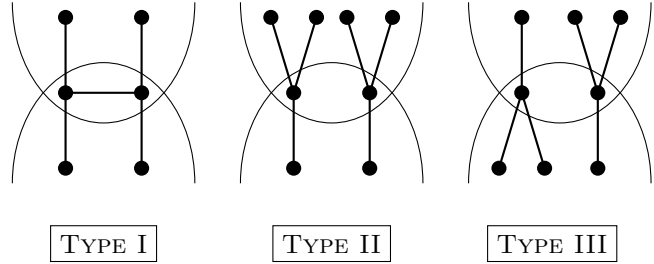
\begin{figure}[ht]
		\centering
		\begin{tikzpicture}
			\draw (0.5,0) node [circle,fill, inner sep=2pt] {};
			\draw (-0.5,0) node [circle,fill, inner sep=2pt] {};
			\draw (0.2,1) node [circle,fill, inner sep=2pt] {};
			\draw (0.8,1) node [circle,fill, inner sep=2pt] {};
			\draw (-0.8,1) node [circle,fill, inner sep=2pt] {};
			\draw (-0.2,1) node [circle,fill, inner sep=2pt] {};
			\draw (0.5,-1) node [circle,fill, inner sep=2pt] {};
			\draw (-0.5,-1) node [circle,fill, inner sep=2pt] {};

			\draw (2.5,0) node [circle,fill, inner sep=2pt] {};
			\draw (3.5,0) node [circle,fill, inner sep=2pt] {};
			\draw (3.8,1) node [circle,fill, inner sep=2pt] {};
			\draw (3.2,1) node [circle,fill, inner sep=2pt] {};
			\draw (2.8,-1) node [circle,fill, inner sep=2pt] {};
			\draw (2.2,-1) node [circle,fill, inner sep=2pt] {};
			\draw (2.5,1) node [circle,fill, inner sep=2pt] {};
			\draw (3.5,-1) node [circle,fill, inner sep=2pt] {};
			
			\draw (-2.5,0) node [circle,fill, inner sep=2pt] {};
			\draw (-3.5,0) node [circle,fill, inner sep=2pt] {};
			\draw (-2.5,-1) node [circle,fill, inner sep=2pt] {};
			\draw (-2.5,1) node [circle,fill, inner sep=2pt] {};
			\draw (-3.5,1) node [circle,fill, inner sep=2pt] {};
			\draw (-3.5,-1) node [circle,fill, inner sep=2pt] {};
			
			\draw[thick,] (-2.5,0)--(-3.5,0) ;
			\draw[thick,] (-2.5,1)--(-2.5,0) ;
			\draw[thick,] (-3.5,1)--(-3.5,0) ;
			\draw[thick,] (-2.5,-1)--(-2.5,0) ;
			\draw[thick,] (-3.5,-1)--(-3.5,0) ;
			
			\draw[thick,] (-0.5,-1)--(-0.5,0);
			\draw[thick,] (0.5,-1)--(0.5,0);
			
			\draw[thick,](-0.2,1)--(-0.5,0);
			\draw[thick,](-0.8,1)--(-0.5,0);
			\draw[thick,](0.2,1)--(0.5,0);
			\draw[thick,](0.8,1)--(0.5,0);
			
			\draw[thick,](3.5,-1)--(3.5,0);
			\draw[thick,](3.5,0)--(3.8,1);
			\draw[thick,](3.5,0)--(3.2,1);
			\draw[thick,](2.5,1)--(2.5,0);
			\draw[thick,](2.5,0)--(2.8,-1);
			\draw[thick,](2.5,0)--(2.2,-1);

			\node[draw] at (-3,-2) {{\sc{Type I}}};
			\node[draw] at (0,-2) {{\sc{Type II}}};
			\node[draw] at (2.9,-2) {{\sc{Type III}}};

			\draw (1.2,-1.2) arc(0:180:1.2cm and 1.6cm);
			\draw (1.2,1.2) arc(0:-180:1.2cm and 1.6cm);
			
			\draw (4.2,-1.2) arc(0:180:1.2cm and 1.6cm);
			\draw (4.2,1.2) arc(0:-180:1.2cm and 1.6cm);
			
			\draw (-1.8,-1.2) arc(0:180:1.2cm and 1.6cm);
			\draw (-1.8,1.2) arc(0:-180:1.2cm and 1.6cm);
			\begin{scope}[shift={(0.5,0.5)}]
			\end{scope}

		\end{tikzpicture}
		\caption{The three types of splitting $2$-separations in cubic Cayley graphs of connectivity $2$.} \label{cubic1}
	\end{figure}
	
	First, we dismiss the case of {\sc{Type III}} separations. This follows as an easy corollary of the following lemma. 
	
	\begin{lemma}\label{crossing_number_0}
		There is always a splitting $2$-separation of crossing number $0$ in $\Gamma$.
	\end{lemma}
	
	\begin{proof}
		Any {\sc{Type I}} separation has crossing number $0$, as the two vertices of the separator are connected with an edge and thus they are inseparable. Hence we can assume that every splitting $2$-separation is not of {\sc{Type I}}. We show that there is always a splitting $2$-separation $(A,B)$ on $A\cap B= \{x,y\}$ such that there are at least two internally disjoint $(x,y)$-paths in $\Gamma[A]$. This will directly imply that there is no $2$-separation $(C,D)$ crossing such an $(A,B)$ as in that case the single vertex in $C \cap D \cap A$ would separate $x$ and $y$ in $\Gamma[A]$, so $\cn(A,B)=0$.
		
		First, we note that for every $2$-separation $(A,B)$ both $A\setminus B, B\setminus A$ are tight components of $G \setminus (A\cap B)$: otherwise, by the $2$-connectivity of $\Gamma$ there are two tight connected components in either $A\setminus B$ or $B\setminus A$, hence there are two internally disjoint $(x,y)$-paths in $\Gamma[A]$ or $\Gamma[B]$. Then $(A,B)$ or $(B,A)$ is the desired separation. 
		
		Now, consider a splitting separation $(A,B)$ on $A \cap B= \{x,y\}$ such that there exist single vertices separating $x,y$ in $\Gamma[A]$ and let $S_A$ be the (non-empty) set of these cut vertices in $\Gamma[A]$. 
		Let $P$ be a shortest $(x,y)$-path  in $\Gamma[A]$. 
		Then $S_A\subseteq V(P)$. 
		It is easily verified that any two consecutive vertices of $S_A \cup \{x,y\}$ in the ordering inherited by $P$ constitute the separator of a $2$-separation $(C,D)$ in $\Gamma$ nested with $(A,B)$, and suppose w.l.o.g.~that $B \subseteq D$.
		Since $(A,B)$ is a splitting separation, say distinguishing $\omega_1, \omega_2$, there is a separation $(C,D)$ as above that is also splitting, distinguishing $\omega_1, \omega_2$ as well. Then $(C,D)$ is the desired separation as any vertex separating the vertices of $C\cap D$ in $\Gamma[C]$ must also separate $x,y$ in $\Gamma[A]$, contradicting the fact that $S_A\cap (C\setminus D)=\emptyset$.
	\end{proof}
	
	\begin{lemma}\label{Never type III}
		Any tree decomposition of $\Gamma$ as in Lemma~\ref{canonical td} is either of {\sc{Type I}} or {\sc{Type II}}.
	\end{lemma}
	
	\begin{proof}
		Let $(A,B)$ be a {\sc{Type III}} separation with $A \cap B=\{x,y\}$. We can assume that $|N(x) \cap A|=1$ and $|N(x) \cap B|=2$. Let $x'$ be the unique neighbor of $x$ in $A$ and $y'$ be the unique neighbor of $y$ in $B$. 
		Then $(A',B'):=(A\cup \{y'\}\setminus \{x\},B \cup \{x'\}\setminus \{x\})$ is a tight {\sc{Type III}} separation on $\{x',y'\}$, clearly crossing $(A,B)$ and distinguishing efficiently the same ends. The lemma is now a direct consequence of Lemma~\ref{canonical td} and Lemma~\ref{crossing_number_0}.
	\end{proof}
	
	In what follows, $(T,\mathcal V)$ will always be as in Lemma~\ref{canonical td}, either of {\sc{Type I}}  or {\sc{Type II}}  if not explicitly stated otherwise. 
	For a node $t\in V(T)$, we define 
	
	\[
	n(t): = \Gamma\left[\bigcup_{t \in N_T[t]} V_t\right].
	\]
	
	Recall that every adhesion set $V_t\cap V_{t'}$ of $(T,\mathcal V)$ induces the separation $(W_{t\setminus t'},W_{t'\setminus t})$ of $\Gamma$.
	Assume that $(T,\mathcal V)$ and the separations $(W_{t\setminus t'},W_{t'\setminus t})$ it induces are of {\sc{Type II}}. 
	We call such a separation $(W_{t\setminus t'},W_{t'\setminus t})$ \emph{small} if the vertices of the separator $V_t\cap V_{t'}$ have degree $1$  in $W_{t'\setminus t}$ and \emph{big} if they have degree $2$ in $W_{t'\setminus t}$.
	
	One of our main goals towards the general structure of the tree-decomposition of $\Gamma$ is to eventually prove in Lemma~\ref{Disjointness Lemma} that all adhesion sets of $(T,\mathcal V)$ are disjoint. 
	As a preparatory step for that, we need the following Lemma. 
	
	\begin{lemma}\label{Preinvolution}
		Every vertex $u$ belongs in at least one and at most two different adhesion sets of $(T,\mathcal V)$ (as subsets of $V(\Gamma)$ and not as intersections of different pairs of parts). Moreover, for every node $t$ of $T$ and every $t_1,t_2 \in N_T(t)$,  we have $|V_{t_1} \cap V_{t_2}| \leq 1$.
	\end{lemma}
	
	\begin{proof}
		The lower bound for the first assertion of the lemma follows directly from the transitivity of the actions of $G$ on $\Gamma$ and $E(T)$. For the upper bound, let $\{x,u\}$ and $\{y,u\}$ be two adhesion sets of the tree-decomposition meeting on $u$. 
		Since $G$ acts transitively on $E(T)$, there is a $1\neq g \in G$ such that $g\{x,u\}=\{y,u\}$. 
		Observe that since $g\neq 1$, we must have $gx=u$ and $gu=y$, from which we obtain $ux^{-1}u=y$. 
		Since $\{x,u\}$ and $\{y,u\}$ were arbitrary adhesion sets containing $u$, the upper bound follows. 
		
		For the second assertion, we clearly have that $|V_{t_1} \cap V_{t_2}|\leq 2$. Suppose that $|V_{t_1} \cap V_{t_2}|=2$.
		It follows from the definition of a tree-decomposition that $V_{t_1}\cap V_{t_2}\subseteq V_t$ and so $V_{t_1}\cap V_{t_2}$ is a subset of both $V_t\cap V_{t_1}$ and $V_t\cap V_{t_2}$.
		Therefore, we have $V_{t_1} \cap V_{t}=V_{t_2} \cap V_{t}=V_{t_1} \cap V_{t_2}:=S$. We observe that the $2$-connectivity of $\Gamma$ implies that all components of $\Gamma \setminus S$ are tight (and in particular, the one containing $V_t\setminus S$).
		
		Let $T_S$ be the subtree of $T$ whose corresponding parts contain $S$. Then $|V(T_S)|\geq 3$ and assume that $|V(T_S)|\geq 4$. Observe that $\Gamma\setminus S$ then has at least four tight components,
		which contradicts the fact that $\Gamma$ is cubic.
		Hence, $|V(T_S)|=3$ and so $V(T_S)=\{t,t_1,t_2\}$.
		Consequently, we see that $C_1= W_{t_1\setminus t}\setminus S$, $C_2= W_{t_2\setminus t}\setminus S$ and $C_3= (W_{t\setminus t_1}) \setminus (W_{t_2\setminus t})=(W_{t\setminus t_2}) \setminus (W_{t_1\setminus t})$ must be the components of $G\setminus S$, all of them tight. 
		
		This means that both vertices of $S$ must have degree one in each of $C_1,C_2,C_3$, respectively, and that $S$ induces an independent set.
		Since $G$ acts transitively on $\Gamma$ and $E(T)$ and $t$ was an arbitrary node of $t$, it follows that every vertex has degree at most one in every part it belongs in.
		We conclude that every part of $\mathcal V$ induces at most a matching where every pair of vertices in the same adhesion set is unmatched. 
		It easily follows that $\Gamma$ is the disjoint union of two infinite cubic trees, contradicting the fact that $\Gamma$ is connected.
	\end{proof}
	
	Let $H$ be an arbitrary graph with a set $U\subseteq V(H)$ and a subgraph $H'$ of $H$.
	The set $U$ is called \emph{connected in $H'$} if for every pair of vertices $u,u'\in U$ there is a $(u,u')$-path in $H'$.
	
	\begin{lemma}\label{Connectivity Lemma}
		Let $t$ be an arbitrary vertex of $T$.
		Then for every $t'\in N_T(t)$, the following holds:
		\begin{enumerate}[label=\rm(\roman*)]
			\item The adhesion set $V_t\cap V_{t'}$ is connected in at least one of $V_t,V_{t'}$.
			\item $V_t$ is connected in $n(t)$.
		\end{enumerate}
	\end{lemma}
	
	\begin{proof}
		\begin{enumerate}[label=(\roman*)]
			\item Let $V_t\cap V_{t'}= \{u,u'\}$ and $P$ be a path between $u$ and $u'$. Suppose that $P$ is contained in $W_{t\setminus t'}$ and consider the tree decomposition $(T', \mathcal V')$ of $P$ by restricting $(T, \mathcal V)$ on the parts that contain at least two vertices of $P$. Notice that every adhesion set of $(T', \mathcal V')$ has size exactly $2$ and root $T'$ on $t$. Since $T'\subseteq T$, the second assertion of Lemma~\ref{Preinvolution} holds for $(T', \mathcal V')$ as well, therefore we have that every part of $(T', \mathcal V')$ contains at least one new vertex of $P$ compared to its predecessor in the tree-order of $T'$ when rooted on $t$.
			Since $P$ is finite, it follows that $T'$ is finite as well. We eventually find a part $V_s$ of $(T,\mathcal V)$ --- in particular a leaf of $T'$ --- such that $P'=V(P)\cap V_s$ is a subpath of $P$ whose end vertices constitute exactly one of the adhesion sets $S$ of $V_s$.
			Recall that $G$ acts transitively on the set of adhesion sets of $(T,\mathcal V)$.
			Hence, we can map $S$ to $V_t\cap V_{t'}$, say $gS=V_t\cap V_{t'}$. Then $gs\in \{t,t'\}$. Thus, $gP'$ is a $(u,u')$-path that either lies in $V_t$ or $V_t'$.
			
			\item  Since $\Gamma$ is connected, the torso of $V_t$ is a connected graph. 
			The result follows by replacing the virtual edges of a path within the torso of $V_t$ by paths obtained by (i).\qedhere
		\end{enumerate}
	\end{proof}
	
	The next crucial lemma implies that all adhesion sets in $\mathcal N$ are disjoint.
	Recall that $(T,\mathcal V)$ is a tree-decomposition as in Lemma~\ref{canonical td}.
	
	\begin{lemma}\label{Disjointness Lemma}
		Let $t$ be a node of $T$. 
		Then for every $t_1,t_2 \in N_T(t)$,  we have $V_{t_1} \cap V_{t_2}=\emptyset$.
	\end{lemma}
	
	\begin{proof}
		By Lemma~\ref{Preinvolution}, we have that $|V_{t_1} \cap V_{t_2}|\leq 1$. Suppose that $|V_{t_1} \cap V_{t_2}|=1$. 
		Let $V_{t_1}\cap V_t=\{x,y\}$, $V_{t_2}\cap V_t=\{x,z\}$. By Lemma~\ref{Preinvolution}, these are the only adhesion sets of $V_t$ containing $x$.
		We can assume that $(T,\mathcal V)$ is of {\sc{Type II}}: indeed, assume that $(T,\mathcal V)$ is of {\sc{Type I}}. 
		By the tightness of all separations in $\mathcal N$, we have that $x$ has at least one neighbor in each of $V_{t_1}\setminus V_t$ and $V_{t_2}\setminus V_t$ in addition to $y$ and $z$, a contradiction to $\Gamma$ being cubic. 
		Hence, $(T,\mathcal V)$ is of {\sc{Type II}}.
		
		Now, by the transitive action on $E(T)$ we have that $(W_{t\setminus t_1}, W_{t_1\setminus t})$ is either isomorphic to $(W_{t\setminus t_2}, W_{t_2\setminus t})$ or its inverse (as an ordered separation).
		In the latter case, assume there are $g \in G$, $t\in V(T)$, such that 
		\begin{equation}\label{separation}
			(W_{t\setminus t_1}, W_{t_1\setminus t})=(gW_{t_2\setminus t}, gW_{t\setminus t_2}).
		\end{equation}
		Recall the definition of small and big separations from Page 7. We can assume w.l.o.g.~that the above separations are small (a similar situation arises in case they are big separations). 
		We observe that this implies that $\deg_{W_{t_1\setminus t}}(x)=1$, $\deg_{W_{t_2\setminus t}}(x)=2$ and the degree of $x$ in the component of $\Gamma\setminus\{x,y,z\}$ containing $V_t\setminus\{x,y,z\}$ is $0$. 
		By Lemma~\ref{Connectivity Lemma}, there is an $(x,y)$-path $P$ lying completely within $V_{t_1}$, but by~(\ref{separation}) we have that the $(x,y)$-path $gP$ lies within $V_t$, which yields a contradiction to the degree of $x$. 
		
		
		
		Otherwise, 
		for every $t\in V(T)$ there is a $g \in G$ such that 
		\[
		(W_{t\setminus t_1}, W_{t_1\setminus t})=(gW_{t\setminus t_2}, gW_{t_2\setminus t}).
		\]
		\noindent
		Let $C'_1, C'_2, C'_3$ be the components of $\Gamma\setminus \{x,y,z\}$ corresponding to $V_{t_1}, V_{t_2}, V_t$. Then we directly observe that both separations $(W_{t\setminus t_1}, W_{t_1\setminus t})$, $(W_{t\setminus t_2}, W_{t_2\setminus t})$ must be small. The only way for this to happen is when $x$ has degree one in each of $C_1,C_2,C_3$ and therefore degree at most one in $V_{t_1}, V_{t_2}, V_t$ (in other words, every part that contains it). By the transitivity of $\Gamma$ and the fact that $t$ was arbitrary, we conclude that every vertex of $\Gamma$ has degree at most one in every part that contains it, to obtain a contradiction to Lemma~\ref{Connectivity Lemma} exactly as before.
	\end{proof}
	
	Lemma~\ref{Disjointness Lemma} has some important consequences. 
	Combined with Lemma~\ref{Preinvolution}, we immediately obtain the following.
	
	\begin{coro}\label{Neighbourhood Lemma}
		Every vertex $u$ of $\Gamma$ is contained in exactly two parts $t,t'\in V(T)$. In addition, $N_{\Gamma}(u) \subseteq V_t \cup V_{t'}$ and every part is the disjoint union of its adhesion sets.\qed
	\end{coro}
	
	Moreover, let $\{x,y\}$ be an adhesion set. Observe that $xy^{-1}\{x,y\}$ is again an adhesion set containing $x$, so $xy^{-1}\{x,y\}=\{x,y\}$ with $xy^{-1}x=y$. We obtain:
	
	\begin{lemma}\label{Involution lemma}
		For every adhesion set $\{x,y\}$, we have $(xy^{-1})^2=1$.\qed
	\end{lemma}
	
	Lemma~\ref{Involution lemma} implies the following Corollary for the edge stabilizers of $T$.
	
	\begin{coro}\label{Edge stabiliser}
		Let $tt'\in E(T)$. Then $\st_G(V_t\cap V_{t'}) \cong \mathbb Z_2$.\qed
	\end{coro}
	
	Lastly, we will invoke the following folklore Lemma from the well-known theory of tree decompositions into $3$-connected components (see \cite{richter2004decomposing,tutte1966connectivity} as an example) when we argue about the planarity of $\Gamma$ and $G$ in each case that arises.
	
	\begin{lemma}\label{planarity}
		Let $(T, \mathcal V)$ be a tight tree-decomposition of a (locally finite) connected graph $H$ with finite parts and adhesion at most $2$. Then $\Gamma$ is planar if and only if the torso of every part of $(T, \mathcal V)$ is planar. 
	\end{lemma}
	
	\begin{proof}
		The forward implication follows from the fact that the torso of a part in $(T, \mathcal V)$ is a topological minor of $H$: for every virtual edge of the part realized by an adhesion set of size exactly two, there is always a path outside of the part that connects the two vertices of the adhesion set.
		
		For the backward implication, we will inductively embed the planar torsos of the parts of $(T, \mathcal V)$ on the plane by gluing (the torso of) each new bag at the appropriate adhesion set. For the inductive step, we consider the adhesion set $V_t \cap V_{t'} = \{x,y\}$ of the new bag $V_t$ along which it is going to be amalgamated with the already embedded torso of $V_{t'}$. We simply replace the edge $xy$ with a planar embedding of the torso of $V_t$ (restricted inside a thin disk around the embedded $xy$), keeping the edge $xy$ in the so far embedded part of $\Gamma$ depending on whether it is an actual edge of $\Gamma$ or a virtual edge of the torsos of $V_t, V_{t'}$, accordingly. It is straightforward to check that the above inductive strategy to combine the planar embeddings of the torsos along the adhesion sets produces an embedding of $\Gamma$ on the plane, since each virtual edge of the torsos is replaced with a new bag at some inductive step.
	\end{proof}
	
	Our goal in the following sections is to determine the structure of the parts of the tree-decomposition of $\Gamma$ obtained by Corollary~\ref{canonical td} in order to compute their stabilizers and apply Lemma~\ref{bass-serre} or~\ref{subdivided bass-serre}.
	
	\section{Tree-decomposition of {\sc{Type I}}}
	
	In this section, we assume that $(T, \mathcal V)$ is of {\sc{Type I}}. Suppose that $b$ is the label of the edge induced by the adhesion sets of $(T, \mathcal V)$, which by  Lemma~\ref{Involution lemma} is an involution. It will be enough to study two neighboring parts $V_t,V_{t'}$ to obtain the general structure of $(T, \mathcal V)$. In order to simplify this, we can assume w.l.o.g~that $ V_t \cap V_{t'}=\{1,b\}$, 
	so $\st_G(V_t \cap V_{t'})=\langle b \rangle$.
	
	Notice that if $G$ acts on $(T, \mathcal V)$ with inversion, there is an element in $g \in \st_G(V_t \cap V_{t'})=\langle b \rangle$ that inverts $V_t, V_{t'}$. Let us express this easy fact with the following lemma.
	
	\begin{lemma}\label{Inversion Type I}
		$G$ acts with inversion on $(T, \mathcal V)$ if and only if $b$ inverts $V_t$ and $V_{t'}$.\qed
	\end{lemma}
	
	\noindent
	
	\begin{lemma}\label{Type I Bag}
		Every part of $\mathcal V$ induces a finite cycle.
	\end{lemma}
	
	\begin{proof}
		Let $t\in V(T)$. Since every adhesion set induces a connected subgraph, we conclude by Lemma~\ref{Torso Lemma} that $\Gamma[V_t]$ is connected. Moreover, Corollary~\ref{Neighbourhood Lemma} implies that $\Gamma[V_t]$ is $2$-regular. It follows that $\Gamma[V_t]$ is either a finite cycle or a double ray. Recall that by Lemma~\ref{Disjointness Lemma} all adhesion sets are disjoint. The conclusion follows by observing that every vertex of $V_t$ is a cut vertex when $V_t$ induces a double ray and hence, the graph $\Gamma$ is not $2$-connected.
	\end{proof}
	
	It will be clear by Lemma~\ref{planarity} that we will obtain planar Cayley graphs in all subcases. 
	
	\subsection{Two Generators}
	
	Assume that $G=\langle {\color{red}a},{\color{blue}b}\rangle$, where $b$ is an involution.
	We distinguish the following cases depending on the colors of the edges incident to the adhesion sets, depicted as Figure \ref{Type I Cases}.

	\begin{figure}[ht]
		\centering
		\begin{tikzpicture}
			\begin{scope}[decoration={
					markings,
					mark=at position 0.55 with {\arrow{>}}}
				] 
				
				\draw (-4.5,0.5) node {$V_t$};
				\draw (-4.5,-0.5) node {$V_{t'}$};
				
				\draw (0,0.5) node {$V_t$};
				\draw (0,-0.5) node {$V_{t'}$};
				
				\draw[blue,thick,] (2,0)--(1,0) ;
				\draw[red,thick,postaction={decorate}] (2,1)--(2,0) ;
				\draw[red,thick,postaction={decorate}] (1,1)--(1,0) ;
				\draw[red,thick,postaction={decorate}] (2,0)--(2,-1) ;
				\draw[red,thick,postaction={decorate}] (1,0)--(1,-1) ;
				
				\draw (2.7,-1.2) arc(0:180:1.2cm and 1.6cm);
				\draw (2.7,1.2) arc(0:-180:1.2cm and 1.6cm);		
				\draw (2,0) node [circle,fill, inner sep=2pt] {};
				\draw (1,0) node [circle,fill, inner sep=2pt] {};
				\draw (2,-1) node [circle,fill, inner sep=2pt] {};
				\draw (2,1) node [circle,fill, inner sep=2pt] {};
				\draw (1,1) node [circle,fill, inner sep=2pt] {};
				\draw (1,-1) node [circle,fill, inner sep=2pt] {};
				\draw[blue,thick,] (-2.5,0)--(-3.5,0) ;
				\draw[red,thick,postaction={decorate}] (-2.5,0)--(-2.5,1) ;
				\draw[red,thick,postaction={decorate}] (-3.5,1)--(-3.5,0) ;
				\draw[red,thick,postaction={decorate}] (-2.5,-1)--(-2.5,0) ;
				\draw[red,thick,postaction={decorate}] (-3.5,0)--(-3.5,-1) ;
				
				\draw (-1.8,-1.2) arc(0:180:1.2cm and 1.6cm);
				\draw (-1.8,1.2) arc(0:-180:1.2cm and 1.6cm);
				
				\node[draw] at (-3,-2) {Case I};
				\node[draw] at (1.5,-2) {Case II};
				\draw (-2.5,0) node [circle,fill, inner sep=2pt] {};
				\draw (-3.5,0) node [circle,fill, inner sep=2pt] {};
				\draw (-2.5,-1) node [circle,fill, inner sep=2pt] {};
				\draw (-2.5,1) node [circle,fill, inner sep=2pt] {};
				\draw (-3.5,1) node [circle,fill, inner sep=2pt] {};
				\draw (-3.5,-1) node [circle,fill, inner sep=2pt] {};

			\end{scope}
			
		\end{tikzpicture}
		\caption{Cases of {\sc{Type I}} with two generators}\label{Type I Cases}
	\end{figure}
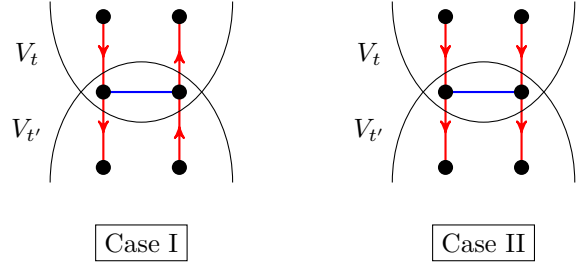        
	
	\subsubsection{Case I}
	
	Suppose that the edges incident to each adhesion set inducing a separation in $\mathcal N$ are as in Case I of Figure~\ref{Type I Cases}. Observe that  $\{a^{-1},ba\}\subseteq V_t$ and  $\{a,ba^{-1}\}\subseteq V_{t'}$ are the neighbors of $1$ and $b$ in $V_t$ and $V_{t'}$, respectively. Since $b\{a^{-1},ba\}=\{a,ba^{-1}\}$, it must be that $bV_t=V_{t'}$ and $bV_{t'}=V_t$. Lemma~\ref{Inversion Type I} implies that $G$ acts on $E(T)$ with inversion (and hence transitively on V(T)).
	
	By Lemma~\ref{Type I Bag}, there is an $n\in \N$ such that $(ba)^n=1$ and $$V_t=\{1,b,ba,\ldots,(ba)^{n-1}b=a^{-1}\}.$$ This gives  a partition $\langle ba\rangle \sqcup \langle ba\rangle b$ of $V_t$. We next conclude that $\st_G(V_t)\subseteq V_t$ by noting that $1\in V_t$.
	Clearly, we have $\langle ba \rangle \subseteq \st_G(V_t)$.
	Moreover, for the element $ba\in V_t$, we observe that 
	
	\[
	(ba)^ib(ba)=(ba)^ia \not\in V_t.
	\]
	Since $V_t=\langle ba\rangle \sqcup \langle ba\rangle b$ , we conclude that $\st_G(V_t)=\langle ba \rangle \cong \mathbb Z_n$. 
	Moreover, $ \st_G(V_t) \cap \st_G(V_t \cap V_{t'})=\langle ba \rangle \cap \langle b \rangle = 1$.
	
	We apply Lemma~\ref{subdivided bass-serre} and obtain that
	
	\[
	G\cong\mathbb Z_n \ast \mathbb Z_2.
	\]
	
	\subsubsection{Case II}
	
	By the structure of the neighbourhood of $\{1,b\}$ and Lemma~\ref{Inversion Type I} we see that $b$ cannot invert $V_t$ and $V_{t'}$, hence $G$ acts on $(T,\mathcal V)$ without inversion.
	
	Now, consider the adhesion set $a^{-1}\{1,b\}=(a^{-1}V_t) \cap (a^{-1}V_{t'})$. From the fact that $a^{-1}\{1,b\} \subseteq V_{t}$ we deduce that either $a^{-1}V_t=V_t$ or $a^{-1}V_{t'}=V_t$. 
	Since the adhesion set $\{1,b\}$ has ingoing $a$-edges but $a\{1,b\}$ has outgoing $a$-edges in $V_t$, we cannot have that $a^{-1}V_t=V_t$. Consequently, it must be that $a^{-1}V_{t'}=V_t$. The fact that two adjacent parts lie in the same orbit under the action of $G$ implies that $G$ acts transitively on $\mathcal V$ (and $V(T))$.
	
	By Lemma~\ref{Type I Bag}, there is in this case an $n\in \N$ such that $(ba^{-1}ba)^n=1$ and
	\[
	V_t=\{1,b,ba^{-1},ba^{-1}b,\ldots,(ba^{-1}ba)^{n-1}ba^{-1}b=a^{-1}\}.
	\]
	In other words, $\langle ba^{-1}ba\rangle \sqcup \langle ba^{-1}ba\rangle b\sqcup \langle ba^{-1}ba\rangle ba^{-1}\sqcup \langle ba^{-1}ba\rangle ba^{-1}b$ forms a partition of $V_t$. 
	Notice that $\langle ba^{-1}ba \rangle$ is the trivial group when $ba^{-1}ba=1$. As before, since $1\in V_t$ we infer that $\st_G(V_t) \subseteq V_t$.
	Clearly, we have $\langle ba^{-1}ba \rangle \subseteq \st_G(V_t)$. 
	Moreover, we see that $\langle ba^{-1}ba \rangle ba^{-1} \not\subseteq \st_G(V_t)$ because we have $(ba^{-1}ba)^i ba^{-1}(ba^{-1}ba)\not\in V_t$ and that $\langle ba^{-1}ba \rangle ba^{-1}a \not\subseteq \st_G(V_t)$ because $(ba^{-1}ba)^i ba^{-1}b(a^{-1}ba) \not\in V_t$. 
	
	Lastly, observe that since $b$ is an involution and all adhesion sets induce a $b$-edge, we have that the action of $b$ on $\Gamma$ fixes every adhesion set. 
	Hence, we have that $b\in \st_G(V_t)$. 
	It follows that $\langle ba^{-1}ba, b \rangle \subseteq \st_G(V_t)$. 
	Therefore, we conclude that 
	$$\st_G(V_t)=\langle ba^{-1}ba, b \mid b^2,(ba^{-1}ba)^n,(a^{-1}ba)^2\rangle \cong D_{2n}.$$
	By Lemma~\ref{bass-serre}, we have that
	
	\[
	G\cong\ D_{2n} \underset{\mathbb Z_2}{\ast}(t).
	\]
	
	\noindent
	We collect both cases in the following theorem.
	
	\begin{thm}\label{Type I two gen}
		If $(T,\mathcal V)$ is of {\sc{Type I}}  with two generators, then $G$ satisfies one of the following cases:
		\begin{enumerate}[label=\rm(\roman*)]
			\item $	G\cong\mathbb Z_n \ast \mathbb Z_2$.
			\item $G\cong D_{2n} \underset{\mathbb Z_2}{\ast}(t)$.\qed
		\end{enumerate}
	\end{thm}
	
	The definitions of a free product with amalgamation, an {\rm HNN}-extension and the proof of Theorem~\ref{Type I two gen} immediately imply:
	
	\begin{coro}{\rm\cite[Theorem 1.1]{Ageloscubic}}
		If $(T,\mathcal V)$ is of {\sc{Type I}} with two generators, then $G$ has one of the following presentations:
		\begin{enumerate}[label=\rm(\roman*)]
			\item $G=\langle a,b\mid b^2,(ba)^n\rangle$.
			\item $G=\langle a,b\mid b^2,(ba^{-1}ba)^n\rangle$.\qed
		\end{enumerate}
	\end{coro}
	
	\subsection{Three Generators}
	
	Let $G=\langle {\color{red}a},{\color{blue}b},{\color{green}c} \rangle$, where $a,b$ and $c$ are involutions.
	Suppose that the edges induced by the adhesion sets with corresponding separations in $\mathcal N$ are colored with $b$. 
	Up to rearranging $a,b,c$, there are two cases for the local structure of the adhesion sets with separations in $\mathcal N$, as in the following figure:
	\begin{figure}[ht]
		\centering
		\begin{tikzpicture}
			\begin{scope}[decoration={
					markings,
					mark=at position 0.55 with {\arrow{>}}}
				] 
				
				\draw (-4.5,0.5) node {$V_t$};
				\draw (-4.5,-0.5) node {$V_{t'}$};
				
				\draw (0,0.5) node {$V_t$};
				\draw (0,-0.5) node {$V_{t'}$};
				
				\draw[blue,thick] (2,0)--(1,0) ;
				\draw[green,thick] (2,0)--(2,1) ;
				\draw[red,thick,] (1,0)--(1,1) ;
				\draw[red,thick,] (2,0)--(2,-1) ;
				\draw[green,thick,] (1,0)--(1,-1) ;
				
				\draw (2,0) node [circle,fill, inner sep=2pt] {};
				\draw (1,0) node [circle,fill, inner sep=2pt] {};
				\draw (2,-1) node [circle,fill, inner sep=2pt] {};
				\draw (2,1) node [circle,fill, inner sep=2pt] {};
				\draw (1,1) node [circle,fill, inner sep=2pt] {};
				\draw (1,-1) node [circle,fill, inner sep=2pt] {};
				
				\draw (2.7,-1.2) arc(0:180:1.2cm and 1.6cm);
				\draw (2.7,1.2) arc(0:-180:1.2cm and 1.6cm);

				\draw[blue,thick,] (-2.5,0)--(-3.5,0) ;
				\draw[red,thick, ](-2.5,0)--(-2.5,1) ;
				\draw[red,thick,] (-3.5,0)--(-3.5,1) ;
				\draw[green,thick,] (-2.5,0)--(-2.5,-1) ;
				\draw[green,thick,] (-3.5,0)--(-3.5,-1) ;
				
				\draw (-1.8,-1.2) arc(0:180:1.2cm and 1.6cm);
				\draw (-1.8,1.2) arc(0:-180:1.2cm and 1.6cm);

				\draw (-2.5,0) node [circle,fill, inner sep=2pt] {};
				\draw (-3.5,0) node [circle,fill, inner sep=2pt] {};
				\draw (-2.5,-1) node [circle,fill, inner sep=2pt] {};
				\draw (-2.5,1) node [circle,fill, inner sep=2pt] {};
				\draw (-3.5,1) node [circle,fill, inner sep=2pt] {};
				\draw (-3.5,-1) node [circle,fill, inner sep=2pt] {};
				
				\node[draw] at (-3,-2) {Case I};
				\node[draw] at (1.5,-2) {Case II};

			\end{scope}
			
		\end{tikzpicture}
		\caption{Cases of {\sc{Type I}}  with three generators}\label{TypeI3Cases}
	\end{figure}
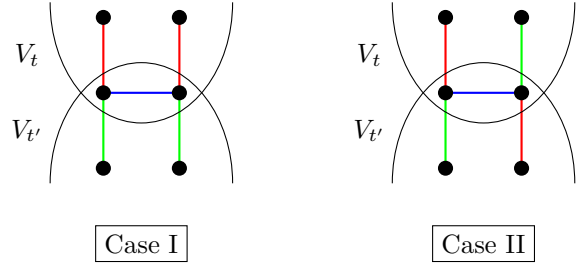   
	
	\subsubsection{Case I}
	
	First, we observe by Lemma~\ref{Inversion Type I} that $G$ acts on $T$ without inversion, since by the structure of the neighbourhood of $\{1,b\}$ we see that $b$ must stabilize both $V_t$ and $V_{t'}$. Consequently, $G$ must act with two orbits $O_1, O_2$ on $\Gamma[\mathcal V]$, where the parts in $O_1$ contain the $a$-edges and the parts in $O_2$ contain the $c$-edges.
	By Lemma~\ref{Type I Bag}  we deduce that there exist $n,m$ with $(ba)^n=1$ and $(bc)^m=1$ and so $V_t=\langle ba\rangle \sqcup \langle ba\rangle b $ and $V_{t'}=\langle bc\rangle \sqcup \langle bc\rangle b $

	To compute the stabilizers of the parts, observe that we can escape a part in $O_1$ only with $c$-edges. 
	Hence, we have $\st_G(V_t)=V_t= \langle  ba,b\mid b^2,(ba)^n,a^2\rangle \cong D_{2n}$ and similarly 
	$\st_G(V_{t'})=V_{t'} =\langle  bc,b\mid b^2,(bc)^m,c^2\rangle \cong D_{2m}$.
	Therefore, by Lemma~\ref{bass-serre} we obtain
	
	\[
	G \cong D_{2n} \underset{\mathbb Z_2}{\ast} D_{2m}.
	\]
	
	\subsubsection{Case II}
	
	In this case, we see that $b$ inverts $V_t$ and $V_{t'}$, so $G$ acts on $T$ with inversion by Lemma~\ref{Inversion Type I}. Hence, $G$ also acts transitively on $V(T)$.
	
	Let $x:=bcba$.
	By Lemma~\ref{Type I Bag} we see that $(bcba)^n=1$ and that $\langle x\rangle \sqcup \langle x\rangle b\sqcup \langle x\rangle bc\sqcup \langle x\rangle bcb$ is a partition of $V_t$. Clearly, we have that $\langle bcba \rangle \subseteq \st_G(V_t)$. We show that we actually have equality:
	\begin{itemize}
		\item $x^ib\cdot bc=x^ic \not\in V_t$, hence $\langle x\rangle b\not\in \st_G(V_t)$,
		\item $x^ibc \cdot a \not\in V_t$, hence $\langle x\rangle bc\not\in \st_G(V_t)$,
		\item $x^ibcb \cdot c \not\in V_t$, hence $\langle x\rangle bcb\not\in \st_G(V_t)$.
	\end{itemize}
	
	We conclude that $\st_G(t) = \langle bcba \rangle \cong \mathbb Z_n$ and consequently we also have that $\st_G(V_t) \cap \st_G(V_t \cap V_{t'})=\langle bcba \rangle \cap \langle b \rangle=1$. It follows from Lemma~\ref{subdivided bass-serre} that
	\[
	G\cong\mathbb Z_n \ast \mathbb Z_2.
	\]
	
	\noindent 
	In conclusion, we have proved:
	
	\begin{thm}\label{Type I three gen}
		If $(T,\mathcal V)$ is of {\sc{Type I}}  with three generators, then $G$ satisfies one of the following cases:
		\begin{enumerate}[label=\rm(\roman*)]
			\item $G \cong D_{2n} \underset{\mathbb Z_2}{\ast} D_{2m}$.
			\item $G\cong\mathbb Z_n \ast \mathbb Z_2$.\qed
		\end{enumerate}
	\end{thm}
	
	\noindent 
	
	\begin{coro}{\rm\cite[Theorem 1.1]{Ageloscubic}}
		If $(T,\mathcal V)$ is of {\sc{Type I}} with three generators, then $G$ has one of the following presentations:
		\begin{enumerate}[label=\rm(\roman*)]
			\item $G=\langle a,b,c\mid a^2,b^2,c^2,(ba)^n,(bc)^m\rangle$.
			\item $G=\langle a,b,c\mid a^2,b^2,c^2,(bcba)^n\rangle$. \qed
		\end{enumerate}
	\end{coro}
	
	\section{Tree-decomposition of {\sc{Type II}}}
	
	Even though at first glance there can be several cases for {\sc{Type II}} separations, we will in fact be able to quickly exclude most of them using appropriately the following lemma.
	
	\begin{lemma}\label{Two-Path Lemma}
		Let $G=\langle {\color{red}a},{\color{blue}b},{\textcolor{green}{c}}\rangle$ (with possibly $c=a^{-1}$), where $b$ is an involution and let $\{x,y\}$ be the adhesion set of a {\sc{Type II}} separation in $(T,\mathcal V)$ of $\Gamma$ as in Lemma~\ref{canonical td}. Let $v_1,v_2,v_3$ be any consecutive vertices in a shortest $(x,y)$-path $P$ with at least two edges and suppose there is $g\in G$ such that $gv_2\in \{x,y\}$. Then $gv_1$ and $gv_3$ lie in the same component of $\Gamma \setminus \{x,y\}$. 
	\end{lemma}
	
	\begin{proof}
		Suppose not. We observe that $gx,gy$ must then lie in different components of $\Gamma \setminus \{x,y\}$ as well: if not, then $gx,gy$ lie in the same component, say, $C$. The fact that $gv_1,gv_3$ lie in different components implies that $gP$ leaves $C$, therefore we have that both $x,y \in V(gP)$. Since $gv_2\in \{x,y\}$ is an inner vertex of $gP$, the subpath of $gP$ from $x$ to $y$ contradicts the choice of $P$.
		
		Hence, $g\{x,y\}$ is a separator where $gx,gy$ lie in different components of $\Gamma \setminus \{x,y\}$. It easily follows that $\{x,y\}$ and $\{gx,gy\}$ are not nested, a contradiction to Lemma~\ref{nested}.
	\end{proof}
	
	Now, let $W_{2n+1,2k}$, where $n\geq 1, k\geq 3, n\leq k$, denote the cubic graph obtained by the ${2k}$-cycle with vertices \{0,1,\ldots,2k-1\} along with ``chord" edges of the form  $\{2i,2i+2n+1\} \pmod{2k}$ forming a matching. Moreover, we define $V_{2n}, n\geq 2$ as the cubic graph obtained by the ${2n}$-cycle along with the ``diagonal" edges $\{i,i+n\} \pmod{2n}$. 
	We note that for $k=2n+1$ we have $W_{2n+1,4n+2}=V_{4n+2}$ (Fig.~\ref{V10}). 
	\begin{figure}[H]
		\hspace*{0.1\linewidth}
		\rule[-1cm]{0pt}{1cm}
		\begin{subfigure}[b]{0.45\textwidth} 
			\begin{tikzpicture}
				\foreach \i in {1,...,10} {
					\node[circle,fill, inner sep=2pt] (\i) at (\i*36:2) {}; 
				}

				\draw[thick] (1) -- (2) -- (3) -- (4) -- (5) -- (6) -- (7) -- (8) -- (9) -- (10) -- (1);
				
				\draw[thick] (1) -- (6);
				\draw[thick] (2) -- (7);
				\draw[thick] (3) -- (8);
				\draw[thick] (4) -- (9);
				\draw[thick] (5) -- (10);
			\end{tikzpicture}
		\end{subfigure}	
		\begin{subfigure}[b]{0.45\textwidth}
			\begin{tikzpicture}
				\foreach \i in {1,...,8} {
					\node[circle,fill, inner sep=2pt] (\i) at (\i*45:2) {}; 
				}
				\draw[thick] (1) -- (2) -- (3) -- (4) -- (5) -- (6) -- (7) -- (8)  -- (1);
				
				\draw[thick] (1) -- (6);
				\draw[thick] (2) -- (5);
				\draw[thick] (3) -- (8);
				\draw[thick] (4) -- (7);
			\end{tikzpicture}
		\end{subfigure}
		\caption{The graphs $W_{5,10}=V_{10}$ and $W_{5,8}$.} \label{V10}

	\end{figure}
	
	Lastly, let $R_{2m+1}$ be the cubic graph obtained by a double ray with vertex set $\mathbb{Z}$ (defined in the natural way) and by adding the edges of the form $\{2i,2i+2m+1\}$ (Fig.~\ref{R5}).
	\begin{figure}[H]
		\centering
		\begin{tikzpicture}
			\foreach \i in {-4,...,5}{
				\draw (\i,1) node [circle,fill, inner sep=2pt] {};
			}
			
			\draw[thick,<->,>=latex'] (-5,1)--(6,1);

			\foreach \i in {-4,-2,0}{
				\draw (\i,1) arc(180:0:2.5cm and 1.2cm);
			}
			
			\draw (2,1) arc(180:50:2.5cm and 1.2cm);
			\draw (4,1) arc(180:100:2.5cm and 1.2cm);
			
			\draw (-1,1) arc(0:130:2.5cm and 1.2cm);
			\draw (-3,1) arc(0:80:2.5cm and 1.2cm);
			
		\end{tikzpicture}
		\caption{The graph $R_5$.} \label{R5}
	\end{figure}
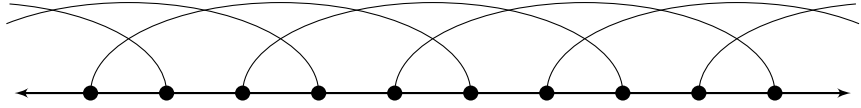
	We note that we will see in the next subsections that the tree-decomposition of $\Gamma$ obtained by Lemma~\ref{canonical td} will have two orbits of parts and that the torsos of the parts of one of the two orbits will always be isomorphic to either $W_{2n+1,2k}$, $V_{2n}$ or $R_{2n+1}$. The fact that $W_{2n+1,2k}$ is never planar, whereas $V_{2n}$ and $R_{2m+1}$ are planar if and only if $n=2$ and $m=1$, respectively, will allow us by Lemma~\ref{planarity} to determine exactly when $\Gamma$ will be planar.
	
	\subsection{Two generators}\label{two generators}
	Let $G=\langle {\color{red}a},{\color{blue}b}\rangle$, where $b$ is an involution. Let $(T,\mathcal V)$ the corresponding tree-decomposition obtained by Lemma~\ref{canonical td} and $\mathcal N$ the set of separations obtained from its adhesion sets.	Then we have the following cases for the neighborhood of such an adhesion set $\{x,y\}$:
	
	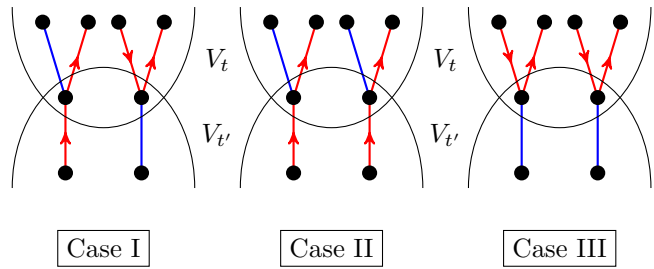
\begin{figure}[ht]
		\centering
		\begin{tikzpicture}
			\begin{scope}[decoration={
					markings,
					mark=at position 0.59 with {\arrow{>}}}
				] 
				
				\draw (1.5,0.5) node {$V_t$};
				\draw (1.5,-0.5) node {$V_{t'}$};
				
				\draw (4.5,0.5) node {$V_t$};
				\draw (4.5,-0.5) node {$V_{t'}$};
				
				\draw[blue,thick] (6.5,-1)--(6.5,0);
				\draw[blue,thick] (5.5,-1)--(5.5,0);
				
				\draw[red,thick,postaction={decorate}](5.5,0)--(5.8,1);
				\draw[red,thick,postaction={decorate}](5.2,1)--(5.5,0);
				\draw[red,thick,postaction={decorate}](6.2,1)--(6.5,0);
				\draw[red,thick,postaction={decorate}](6.5,0)--(6.8,1);
				
				\draw (7.2,-1.2) arc(0:180:1.2cm and 1.6cm);
				\draw (7.2,1.2) arc(0:-180:1.2cm and 1.6cm);
				
				\draw (6.5,0) node [circle,fill, inner sep=2pt] {};
				\draw (5.5,0) node [circle,fill, inner sep=2pt] {};
				\draw (6.2,1) node [circle,fill, inner sep=2pt] {};
				\draw (6.8,1) node [circle,fill, inner sep=2pt] {};
				\draw (5.2,1) node [circle,fill, inner sep=2pt] {};
				\draw (5.8,1) node [circle,fill, inner sep=2pt] {};
				\draw (6.5,-1) node [circle,fill, inner sep=2pt] {};
				\draw (5.5,-1) node [circle,fill, inner sep=2pt] {};
				
				
				\draw[red,thick,postaction={decorate}] (2.5,-1)--(2.5,0);
				\draw[red,thick] (3.5,-1)--(3.5,0);
				
				\draw[red,thick,postaction={decorate}] (2.5,0)--(2.8,1);
				\draw[blue,thick] (2.2,1)--(2.5,0);
				\draw[blue,thick] (3.2,1)--(3.5,0);
				\draw[red,thick](3.5,0)--(3.8,1);
				
				\draw (4.2,-1.2) arc(0:180:1.2cm and 1.6cm);
				\draw (4.2,1.2) arc(0:-180:1.2cm and 1.6cm);
				
				\draw (3.5,0) node [circle,fill, inner sep=2pt] {};
				\draw (2.5,0) node [circle,fill, inner sep=2pt] {};
				\draw (3.2,1) node [circle,fill, inner sep=2pt] {};
				\draw (3.8,1) node [circle,fill, inner sep=2pt] {};
				\draw (2.2,1) node [circle,fill, inner sep=2pt] {};
				\draw (2.8,1) node [circle,fill, inner sep=2pt] {};
				\draw (3.5,-1) node [circle,fill, inner sep=2pt] {};
				\draw (2.5,-1) node [circle,fill, inner sep=2pt] {};


				\draw[red,thick,postaction={decorate}](-0.5,-1)--(-0.5,0);
				\draw[blue,thick] (0.5,-1)--(0.5,0);
				
				\draw[red,thick,postaction={decorate}](-0.5,0)--(-0.2,1);
				\draw[blue,thick](-0.8,1)--(-0.5,0);
				\draw[red,thick,postaction={decorate}](0.2,1)--(0.5,0);
				\draw[red,thick,postaction={decorate}](0.5,0)--(0.8,1);
				
				\node[draw] at (0,-2) {Case I};
				\node[draw] at (3,-2) {Case II};
				\node[draw] at (6,-2) {Case III};

				\draw (1.2,-1.2) arc(0:180:1.2cm and 1.6cm);
				\draw (1.2,1.2) arc(0:-180:1.2cm and 1.6cm);

				\draw (0.5,0) node [circle,fill, inner sep=2pt] {};
				\draw (-0.5,0) node [circle,fill, inner sep=2pt] {};
				\draw (0.2,1) node [circle,fill, inner sep=2pt] {};
				\draw (0.8,1) node [circle,fill, inner sep=2pt] {};
				\draw (-0.8,1) node [circle,fill, inner sep=2pt] {};
				\draw (-0.2,1) node [circle,fill, inner sep=2pt] {};
				\draw (0.5,-1) node [circle,fill, inner sep=2pt] {};
				\draw (-0.5,-1) node [circle,fill, inner sep=2pt] {};
				
			\end{scope}
			
		\end{tikzpicture}
		\caption{Cases of {\sc{Type II}} with two generators.} \label{typeII2cases}
	\end{figure}
	
	\begin{lemma}\label{CaseIII}
		The adhesion sets of $(T,\mathcal V)$ satisfy Case III.
	\end{lemma}
	
	\begin{proof}
		Let $\{x,y\}$ be an adhesion set. First, observe that no path in $\Gamma$ contains two consecutive $b$-edges, hence every path of length two contains at least one $a$-edge. Let $P$ be a shortest $(x,y)$-path\footnote{By Lemma~\ref{Connectivity Lemma}(i) we can see that $P$ lies completely within $V_{t}$ or $V_{t'}$, but this is irrelevant to the proof of the Lemma.}, necessarily of length at least two. 
		
		Assume that either Case I or Case II happen. Notice that --in both cases-- for every possible edge-coloring of a path of length two there exists a path $Q$ of length two whose middle vertex belongs in $\{x,y\}$ and its two endpoints lie in different components of $\Gamma \setminus \{x,y\}$ that realizes the same edge-coloring. Consider an arbitrary subpath $P'=v_1v_2v_3$ of $P$ of length two and an appropriate $Q$ as above that realizes the edge-coloring of $P'$. Let $w$ be the middle vertex of $Q$ and $g=wv_2^{-1}$. Then $gP=Q$ and $gv_1,gv_3$ lie in different components of $\Gamma \setminus \{x,y\}$, contradicting Lemma~\ref{Two-Path Lemma}.
	\end{proof}
	
	Consequently, we can assume for the rest of this subsection that only Case III happens. It follows that no part of $(T,\mathcal V)$ contains edges of all colors: otherwise, by Corollary~\ref{Neighbourhood Lemma} and the fact that no adhesion set contains both a vertex incident with $a$-edges as well as a vertex incident with $b$-edges in a part $V_t$, we see that the $a$-edges and the $b$-edges induce different connected components in the torso of $V_t$, a contradiction to the connectivity of $\Gamma$.
	
	Hence, $(T,\mathcal V)$ has two orbits of parts $O_1,O_2$, where parts in $O_1$ contain only edges colored with $a$ and parts in $O_2$ contain edges colored with $b$. Moreover, $G$ acts on $(T,\mathcal V)$ without inversion. The structure of the parts in $O_2$ is clear: their edges induce a perfect $b$-matching in the part. We are ready to obtain the full structure of the parts in $O_1$ as well.
	
	\begin{lemma}\label{a-bag}
		There is an $n\geq 2$, such that for every adhesion set $\{x,y\}$ we have $x=ya^n$ or $x=ya^{-n}$. Moreover, every part in $O_1$ induces an $a$-cycle of length $2n$.
	\end{lemma}
	
	\begin{proof}
		Let $V_t\in O_1$ and $\{x,y\}=V_t\cap V_{t'}$ be an adhesion set of $t$. For every $s\in N_T(t)$, we have that $V_s \in O_2$ and consequently that $V_s$ induces a $b$-matching. By Lemma~\ref{Connectivity Lemma}(ii), it follows that $\Gamma[V_t]$ is connected. 
		
		Consider an $(x,y)$-path $P$ within $V_t$ and let $n\geq 2$ be its length. Hence, $x=ya^n$ or $x=ya^{-n}$. By Lemma~\ref{Involution lemma}, we have $(xy^{-1})^2=1$, from which we obtain $a^{2n}=1$ after substituting $x$.
		
		We have inferred that the $2$-regular graph $\Gamma[V_t]$ is connected. 
		Notice that $P,xy^{-1}P$ are internally disjoint paths with the same ends living inside $V_t$, therefore their concatenation induces a cycle in $V_t$. Recall that $a$ has order $2n$. This directly implies the Lemma.  
	\end{proof}
	
	Observe that the torso of a part $V_s\in O_2$ induces a connected, $2$-regular graph. It cannot be a double ray: in that case every vertex is a cut vertex (as is easily seen), which violates the $2$-connectivity of $\Gamma$. Hence, the torso of $V_s$ induces a finite cycle, whose edges we can label by Lemma~\ref{a-bag} with $a^n$ (corresponding to the virtual edges of the torso) and $b$ in an alternating fashion. Therefore, there is an $m \geq 2$ such that $(ba^n)^m=1$. 
	
	
	It remains to compute the vertex stabilizers of $T$.
	
	Let $V_{t_1}\in O_1$ such that $1 \in V_{t_1}$. By Lemma~\ref{a-bag}, we clearly have 
	$\langle a \rangle=V_{t_1}$ and therefore $\st_G(V_{t_1})=\langle a \rangle\cong \mathbb Z_{2n}$.
	Next, let $V_{t_2} \in O_2$ such that $1 \in V_{t_2}$. Recall that $(ba^n)^m=1$ and notice that $(b(ba^n))^2 =a^{2n}=1$. By the structure of the torso of $V_{t_2}$, we observe that the elements of $V_{t_2}$ form a group generated by $b$ and $ba^n$ with presentation $\langle ba^n,b \mid ((ba)^n)^m,b^2, (b(ba^n))^2\rangle$. Since $V_{t_2}$ forms a subgroup of $G$, we deduce that 
	\[
	\st_G(V_{t_2})=V_{t_2}=\langle ba^n,b \mid ((ba)^n)^m,b^2, (b(ba^n))^2\rangle\cong D_{2m}.
	\]
	\noindent 
	Finally, by Lemma~\ref{bass-serre} we obtain $G\cong\mathbb Z_{2n} \underset{\mathbb Z_2}{\ast} D_{2m}$.
	
	We observe that the torso of $V_{t_1}$ is isomorphic to $V_{2n}$. Since $V_{2n}$ is planar if and only if $n=2$, we conclude by Lemma~\ref{planarity} that $\Gamma$ is planar if and only if $n=2$. We have obtained the following theorem, along with its corollary by the definition of a free product with amalgamation:
	
	\begin{thm}\label{Type II two gen}
		If $(T,\mathcal V)$ is of {\sc{Type II}}  with two generators, then
		\[
		G\cong\mathbb Z_{2n} \underset{\mathbb Z_2}{\ast} D_{2m}.
		\]
		In particular, $G$ is planar if and only if $n=2$.\qedhere
	\end{thm}
	
	\begin{coro}{\rm\cite[Theorem 1.1]{Ageloscubic}}
		If $(T,\mathcal V)$ is of {\sc{Type II}} with two generators, then \[
		G=\langle a,b \mid b^2, a^{2n}, (ba^n)^m \rangle.
		\]
		In particular, $G$ is planar if and only if $n=2$.\qedhere
	\end{coro}

	\subsection{Three generators}
	Let $G=\langle {\color{red}a},{\color{blue}b},{\textcolor{green}{c}}\rangle$, where $a,b$ and $c$ are involutions.
	Then --up to rearranging $a,b,c$-- we have the following cases for the separations in $\mathcal N$:
	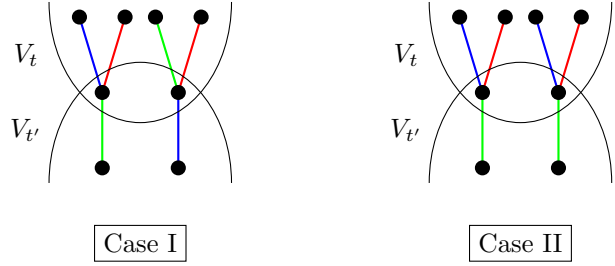
\begin{figure}[ht]
		\centering
		\begin{tikzpicture}

			\draw (-1.5,0.5) node {$V_t$};
			\draw (-1.5,-0.5) node {$V_{t'}$};
			
			\draw (3.5,0.5) node {$V_t$};
			\draw (3.5,-0.5) node {$V_{t'}$};
			
			\draw[red,thick] (5.5,-1)--(5.5,0);
			\draw [red,thick](4.5,-1)--(4.5,0);
			
			\draw[blue,thick](4.8,1)--(4.5,0);
			\draw[green,thick](4.2,1)--(4.5,0);
			\draw[green,thick](5.2,1)--(5.5,0);
			\draw[blue, thick](5.8,1)--(5.5,0);
			
			\draw (6.2,-1.2) arc(0:180:1.2cm and 1.6cm);
			\draw (6.2,1.2) arc(0:-180:1.2cm and 1.6cm);
			
			\draw (5.5,0) node [circle,fill, inner sep=2pt] {};
			\draw (4.5,0) node [circle,fill, inner sep=2pt] {};
			\draw (5.2,1) node [circle,fill, inner sep=2pt] {};
			\draw (5.8,1) node [circle,fill, inner sep=2pt] {};
			\draw (4.2,1) node [circle,fill, inner sep=2pt] {};
			\draw (4.8,1) node [circle,fill, inner sep=2pt] {};
			\draw (5.5,-1) node [circle,fill, inner sep=2pt] {};
			\draw (4.5,-1) node [circle,fill, inner sep=2pt] {};

			\draw[green,thick] (-0.5,-1)--(-0.5,0);
			\draw[blue,thick] (0.5,-1)--(0.5,0);
			\draw[red,thick](-0.2,1)--(-0.5,0);
			\draw[blue,thick](-0.8,1)--(-0.5,0);
			\draw[green,thick](0.2,1)--(0.5,0);
			\draw[red,thick](0.8,1)--(0.5,0);
			
			\draw (0.5,0) node [circle,fill, inner sep=2pt] {};
			\draw (-0.5,0) node [circle,fill, inner sep=2pt] {};
			\draw (0.2,1) node [circle,fill, inner sep=2pt] {};
			\draw (0.8,1) node [circle,fill, inner sep=2pt] {};
			\draw (-0.8,1) node [circle,fill, inner sep=2pt] {};
			\draw (-0.2,1) node [circle,fill, inner sep=2pt] {};
			\draw (0.5,-1) node [circle,fill, inner sep=2pt] {};
			\draw (-0.5,-1) node [circle,fill, inner sep=2pt] {};
			
			\node[draw] at (0,-2) {Case I};
			\node[draw] at (5,-2) {Case II};
			
			\draw (1.2,-1.2) arc(0:180:1.2cm and 1.6cm);
			\draw (1.2,1.2) arc(0:-180:1.2cm and 1.6cm);

			\begin{scope}[shift={(0.5,0.5)}]
			\end{scope}

		\end{tikzpicture}
		\caption{{\sc{Type II}}  cases with three generators} \label{typeII3cases}
	\end{figure}
	
	As in Subsection~\ref{two generators}, by properly applying Lemma~\ref{Two-Path Lemma} we obtain the analogue of Lemma~\ref{CaseIII} for three generators with exactly the same proof.
	
	\begin{lemma}\label{CaseII}
		The adhesion sets of $(T,\mathcal V)$ satisfy Case II. \qed
	\end{lemma}
	
	Since the torso of every part of $(T,\mathcal V)$ is a connected graph, we deduce that the tree-decomposition has two orbits of parts: parts in $O_1$ contain only $b$- and $c$-edges and parts in $O_2$ induce perfect $a$-matchings. Clearly, $G$ then acts on $(T, \mathcal V)$ without inversion. Let us quickly obtain the analogue of Lemma~\ref{a-bag}.
	
	\begin{lemma}\label{bc-bag}
		Every part in $O_1$ induces an alternating $(b,c)$-cycle of even length or an alternating  double $(b,c)$-ray.
	\end{lemma}
	
	\begin{proof}
		Let $V_t\in O_1$ and $\{x,y\}=V_t\cap V_{t'}$ be an adhesion set of $t$. Since all neighbours of $t$ induce an $a$-matching, it follows by Lemma~\ref{Connectivity Lemma}(ii) that $\Gamma[V_t]$ is connected. 
		
		Hence, there exists an $(x,y)$-path $P$ of length $i$ within $V_t$, necessarily alternating with $b$- and $c$-edges. Then, either $x=y(bc)^n$ or $x=y(bc)^nb$, up to swapping $b$ and $c$. To obtain the structure of the $2$-regular, connected graph $V_t$ we distinguish the following cases.
		
		\begin{itemize}
			\item[(i)] If $x=y(bc)^n$, then the $(x,y)$-path $xy^{-1}P$ intersects $P$ only in $x,y$ and by Lemma~\ref{Involution lemma}, we obtain $(bc)^{2n}=1$. In this case, $V_t$ induces an alternating $(b,c)$-cycle of length $4n$.
		\end{itemize}
		
		If $x=y(bc)^nb$, then $xy^{-1}P=P$ and:
		
		\begin{itemize}
			\item[(ii)] either $V_t$ induces an alternating $(b,c)$-cycle,
			\item[(iii)] or $V_t$ induces an alternating double $(b,c)$-ray.\qedhere
		\end{itemize}
	\end{proof}
	
	By the $2$-connectivity of $\Gamma$, the connected, $2$-regular torso of a part $V_s \in O_2$ must be a finite cycle. Depending on which of the cases of Lemma~\ref{bc-bag} we have, we can label its edges  with $(bc)^n$ or $(bc)^nb$ (corresponding to the virtual edges of the torso) and $a$ in an alternating fashion. Therefore, there is an $m \geq 2$ such that $(a(bc)^n)^m=1$ or $(a(bc)^nb)^m=1$. It remains to infer the structure of $G$ in each case.
	
	\begin{enumerate}[label=(\roman*)]
		\item Assume that every part in $O_1$ is an alternating $(b,c)$-cycle of length $4n$ and $(a(bc)^n)^m=1$.
		
		
		In order to compute the vertex stabilizers of $T$, let $V_{t_1}\in O_1$ with $1 \in V_{t_1}$. Since $(b(bc))^2=c^2=1$, we have that
		\[
		V_{t_1}=\langle bc \rangle \cup \langle bc \rangle b=\langle bc, b \mid (bc)^{2n}, b^2, (b(bc))^2 \rangle\cong D_{4n}.
		\] 
		Then $\st_G(V_{t_1})=V_{t_1}\cong D_{4n}$, as $V_{t_1}$ forms a group.
		Next, let $V_{t_2} \in O_2$ with $1 \in V_{t_2}$. Notice that  $(a(bc)^n)^m=a^2=1$ and $(a(a(bc)^n))^2=(bc)^{2n}=1$. We can deduce that $V_{t_2}$ is a group (and hence $\st_G(V_{t_2})=V_{t_2}$), along with its presentation: 
		\[
		\st_G(V_{t_2})=V_{t_2}=\langle a(bc)^n,a \mid (a(bc)^n)^m, a^2, (a(a(bc)^n))^2\rangle\cong D_{2m}.
		\]
		
		By Lemma~\ref{bass-serre}, we have
		\[
		G\cong D_{4n} \underset{\mathbb Z_2}{\ast} D_{2m}.
		\]
		
		In this case, the torso of $V_{t_1}$ is isomorphic to $V_{4n}$, which is planar if and only if $n=1$. 
		
		\item Suppose that parts in $O_1$ induce an alternating $(b,c)$-cycle of even length and that $(a(bc)^nb)^m=1$. \\
		Let $V_{t_1} \in O_1$ and  $V_{t_2} \in O_2$ both containing $1$ in the respective parts.
		We see that 
		\[
		\st_G(V_{t_1})=V_{t_1}=\langle bc, b \mid (bc)^k ,b^2, (b(bc))^2 \rangle\cong D_{2k},
		\]
		\[
		\st_G(V_{t_2})=V_{t_2}=\langle a(bc)^nb,a \mid (a(bc)^nb)^m, a^2, (a(a(bc)^nb))^2\rangle\cong D_{2m}.
		\]
		By Lemma~\ref{bass-serre},
		\[
		G\cong D_{2k} \underset{\mathbb Z_2}{\ast} D_{2m},
		\]
		
		Notice that the torso of $V_{t_1}$ is isomorphic to $W_{2n+1,2k}$, which is not planar.
		
		\item In this case, every part in $O_1$ is an alternating double $(b,c)$-ray and $(a(bc)^nb)^m=1$. 
		
		Let $V_{t_1} \in O_1$ and  $V_{t_2} \in O_2$ with $1$ contained in their common adhesion set. 
		Similarly, we have that 
		\[
		\st_G(V_{t_1})=V_{t_1}=\langle bc, b \mid b^2, (b(bc))^2 \rangle\cong D_{\infty},
		\]
		\[
		\st_G(V_{t_2})=V_{t_2}=\langle a(bc)^nb,a \mid (a(bc)^nb)^m, a^2, (a(a(bc)^nb))^2\rangle\cong D_{2m}.
		\]
		By Lemma~\ref{bass-serre}, we have
		\[
		G\cong D_{\infty} \underset{\mathbb Z_2}{\ast} D_{2m},
		\]
		
		to conclude that the torso of $V_{t_1}$ is isomorphic to $R_{2n+1}$, which is planar if and only if $n=1$. 
	\end{enumerate}
	\noindent
	
	By Lemma~\ref{planarity} and the above discussion, we have deduced:

	\begin{thm}\label{Type II three gen}
		If $(T,\mathcal V)$ is of {\sc{Type II}}  with three generators, then $G$ satisfies one of the following cases:
		\begin{enumerate}[label=\rm(\roman*)]
			\item $G \cong D_{4n} \underset{\mathbb Z_2}{\ast} D_{2m}$.
			\item $G\cong D_{2k} \underset{\mathbb Z_2}{\ast} D_{2m}$
			\item $G\cong D_{\infty} \underset{\mathbb Z_2}{\ast} D_{2m}$.\qed
		\end{enumerate}
	\end{thm}
	
	\begin{coro}{\rm\cite[Theorem 1.1]{Ageloscubic}}
		If $(T,\mathcal V)$ is of {\sc{Type II}} with three generators, then $G$ has one of the following presentations:
		\begin{enumerate}[ label=\rm(\roman*)]
			\item $G=\langle a,b,c \mid a^2, b^2, c^2, (bc)^{2n}, (a(bc)^n)^m \rangle$ and $\Gamma$ is planar if and only if $n=1$.
			\item $G=\langle a,b,c \mid a^2, b^2, c^2, (bc)^k,(a(bc)^nb)^m \rangle$. $\Gamma$ is not planar.
			\item $G=\langle a,b,c \mid a^2, b^2, c^2, (a(bc)^nb)^m \rangle$ and $\Gamma$ is planar if and only if $n=1$.\qed
		\end{enumerate}
	\end{coro}
	
	\section{Open Questions}
	
	Having obtained the full characterization of groups admitting cubic Cayley graphs of connectivity two, some further open questions can naturally be raised. In light of Lemma~\ref{one-ended},  we can ask the following.
	\begin{prob}
		Characterize all groups admitting $4$-regular Cayley graphs of connectivity at most three in terms of splitting over subgroups. 
	\end{prob}
	
	Let $\mathcal G$ be the family of graphs containing all cycles and all graphs of the form $W_{2n+1,k}$, $V_{2n}$ or $R_{2m+1}$.
	A graph is called quasi-transitive if it has a finite number of orbits under the action of its automorphism group.
	Looking back at Theorem \ref{group characterization}, we see that cubic Cayley graphs of connectivity two can be expressed 
	as a tree decomposition whose torsos induce graphs from $\mathcal G$.
	The main tools from our proof seem to go through to support that this is in general the case for every cubic transitive graph of connectivity two.
	We can go a step further and ask the following question:
	
	\begin{prob}
		Characterize all cubic quasi-transitive graphs of connectivity two in terms of ``canonical'' tree decompositions with the property that the automorphism group of the graph acts transitively on the set of the adhesion sets.
	\end{prob}
	
	\bibliographystyle{plain}
	\bibliography{collective.bib}  
\end{document}